\documentclass[12pt]{article}

\usepackage{amsmath,amsthm,amsfonts,amssymb,color}

\newcommand{\hs}{\hspace{1.0em}}

\newcommand{\SFFD}{\text{II}_D}
\newcommand{\X}{\mathfrak{X}}

\renewcommand{\d}[2]{\frac{d#1}{d#2}}

\def\hook{{\mathchoice{\vrule height 0pt depth 0.4pt width 3pt
\vrule height 5pt depth 0.4pt \kern 3pt} {\vrule height 0pt depth
0.4pt width 3pt \vrule height 5pt depth 0.4pt \kern 3pt} {\vrule
height 0pt depth 0.2pt width 1.5pt \vrule height 3pt depth 0.2pt
width 0.2pt \kern 1pt} {\vrule height 0pt depth 0.2pt width 1.5pt
\vrule height 3pt depth 0.2pt width 0.2pt \kern 1pt} }}

\theoremstyle{plain}
\newtheorem{thm}{Theorem}[section]

\newtheorem{propn}[thm]{Proposition}
\newtheorem{cor}[thm]{Corollary}

\theoremstyle{definition}
\newtheorem{defn}[thm]{Definition}

\theoremstyle{remark}
\newtheorem*{rmk}{Remarks}

\begin{document}

\title{Torsion and the second fundamental form for distributions}


\author{G.E. Prince
\thanks{Email: {\tt geoff\char64 amsi.org.au}}
\\Department of Mathematics and Statistics,\\ La Trobe University,\\ Victoria 3086,
Australia;\\ The Australian Mathematical Sciences Institute,\\ c/o The University of Melbourne,\\Victoria 3010, Australia}

\maketitle

\begin{abstract}
The second fundamental form of Riemannian geometry is generalised to the case of a manifold with a linear connection and an integrable distribution.
This bilinear form is generally not symmetric and its skew part is the torsion. The form itself is closely related to the shape map of the connection. The codimension one case generalises the traditional shape operator of Riemannian geometry.
\end{abstract}

\section{Motivation}
Even though the torsion of a linear connection is intrinsically defined as a vector-valued two-form, it is natural to ask if it is the skew symmetric part of some vector-valued bilinear form. In this short communication we will show, in the context of an integrable distribution, that the part of the torsion transverse to the leaves of the corresponding foliation is indeed the skew part of a bilinear form, namely the second fundamental form of the distribution.

This result can be seen in the broader context of the identification of results apparently restricted to fields such as Riemannian, Finsler and contact geometry as specific cases of more general theorems in the geometries of linear connections and elsewhere.

The monograph of Bejancu and Farran~\cite{BF06} does deal with the second fundamental form on distributions, but in the restricted context of semi-Riemannian geometry where orthogonal complements are available. The symmetry of their second  fundamental form is a necessary and sufficient condition for the integrability of the distribution in question. This is a rather different development to that which we present here and we leave the interested reader to explore the ideas of Bejancu and Farran.

\section{The Riemannian case}

We follow Lee~\cite{Lee97}. Suppose that $S$ is an embedded Riemannian submanifold, of constant dimension $p$, of an $n$-dimensional Riemannian manifold $(M,g).$ Let $\nabla^g$ be the Levi-Civita connection on $M$ and, if $\bar g$ is the restriction of $g$ to $S$, then $\nabla^{\bar g}$ is the Levi-Civita connection of $\bar g$ on $S$. For $x\in S$ let $T_xM=T_xS\bigoplus N_x$ where $N_x$ is the orthogonal complement of $T_xS$. Define $\pi^\top$ and $\pi^\bot$ to be projectors of $T_xM$ onto $T_xS$ and $N_x$ respectively and, for $X$ on $S$, denote $\pi^\top(X)$ by $X^\top$ and $\pi^\bot(X)$ by $X^\bot$.

Setting  $N(S):=\displaystyle\bigcup_{\stackrel{x\in S}{}} N_xS,$ the second fundamental form, II, of $S$ is the map from $\X(S)\times\X(S) \to N(S):$
$$\text{II}(X,Y):=(\nabla^g_XY)^\bot$$
where on the right hand side $X,Y$ are extended arbitrarily to $M$. With this definition Lee establishes the well-known results

\begin{propn} The second fundamental form is
\begin{itemize}
\item[1.] independent of the extensions of $X$ and $Y$;
\item[2.]bilinear over $C^\infty(S)$;
\item[3.]symmetric in $X$ and $Y$; and satisfies
\item[4.] $\nabla^g_XY=\nabla^{\bar g}_XY+\text{II}(X,Y)$\ (the Gauss Formula).
\end{itemize}
\end{propn}

The bilinearity of $\text{II}$ is established using the linearity in the first argument and the symmetry property 3. above. This last property holds because the connection is symmetric, that is, it has zero torsion. As we will see dropping the metric from the picture changes this.

\section{Manifolds with connection}\label{sect3}

Suppose that $M$ is a manifold of dimension $n$ with a linear connection $\nabla$ having non-zero torsion $T$. For $X,Y \in \X(M)$ we have the usual definition
\begin{equation}\label{torsion}
T(X,Y):= \nabla_XY-\nabla_YX-[X,Y].
\end{equation}

The shape map of the connection is an endomorphism of tangent spaces defined as follows. Denoting parallel transport using
$\nabla$ on $M$ by $\tau_t$ and denoting the flow generated by a vector field $Z$ on $M$ by
$\zeta_t,$ we have (see \cite{CP84, JP00})

\begin{defn}\label{shape map}
\begin{equation*}
A_Z(\xi) := \d{}{t}\bigg\vert_{t=0} \tau^{-1}_t(\zeta_{t*}\xi)
\hs\text{where}\hs \xi\in T_xM.
\end{equation*}
\end{defn}

\noindent The shape map is intimately related to the torsion as follows (see \cite{JP02} and also Proposition 2.9 of volume I of Kobayashi and Nomizu~\cite{KN63} which guarantees that the sum $\nabla Z + Z \hook T$ is a tangent space endomorphism).

\begin{propn}\label{torsion AZ}
\begin{equation*}
A_Z(\xi) =  \nabla_\xi Z + T(Z_x,\xi), \hs \xi\in T_xM.
\end{equation*}
\end{propn}

\begin{proof}
Let $X$ be the field obtained by Lie dragging $\xi$ along the
integral curve of $Z$ through $x$. Then
\begin{align*}
A_Z(\xi) &=\d{}{t}\bigg\vert_{t=0} (\tau^{-1}_t X_{\zeta_t(x)}) =
(\nabla_Z
X)_x\\
&= (\nabla_X Z)_x + T(Z,X)_x + (\mathcal{L}_ZX)_x = \nabla_\xi Z + T(Z_x,\xi)
\end{align*}
where we have used \eqref{torsion} and
$\mathcal{L}_ZX=0$.
\end{proof}

This result indicates that $A_Z$ is not in general function linear in $Z$. When the connection is symmetric $A_Z$ is just $\nabla Z$, the
covariant differential of $Z$. The Raychaudhuri
equation and its generalisation are obtained by assuming $Z$ is auto-parallel with respect to
$\nabla$ and taking the trace of $\mathcal{L}_Z A_Z,$ see \cite{CP84,JP00}. Vector fields
satisfying $A_Z X=\nabla_Z X$ along an auto-parallel field $Z$ can be
shown to satisfy a generalised Jacobi's equation~\cite{JP02}.

Along with \eqref{torsion} the following identities will be useful
\begin{align}
A_X(Y)&= \nabla_YX+T(X,Y)\label{SM1} \\
A_X(Y)&=\nabla_XY-[X,Y]\label{SM2}\\
T(X,Y)&=A_X(Y)-A_Y(X)+[X,Y]. \label{SM3}
\end{align}

Now suppose that $D\subset\X(M)$ is an integrable distribution of constant dimension $p$ on $M$ with annihilator $D^\bot\subset\mathfrak{X}^*(M).$ (We don't distinguish these sub-bundles from the submodules that they generate.)
Further suppose that $\X(M)=D\bigoplus D'$ where $D'$ is fixed and not necessarily integrable. Define $\pi^\top$ and $\pi^\bot$ to be projectors of $T_xM$ onto $D_x$ and $D_x'$ respectively, and denote $\pi^\top(X)$ by $X^\top$ and $\pi^\bot(X)$ by $X^\bot$. Note that since $D_x$ is defined at every point $x\in M$ there is no a priori need for two separate connections, one on $M$ and one on $D.$ Following the standard definition of the fundamental form given in the preceding section we define the second fundamental form on $D$ by
\begin{equation}\label{eq:SFFD}
\SFFD(X,Y):=(\nabla_XY)^\bot \quad \text{for}\ X,Y\in D.
\end{equation}

We will now show that $\SFFD$ is a bilinear form on $D$ whose skew symmetric part is $T^\bot.$

\begin{propn}\label{SFFD}
\
For all $X,Y \in D$
\begin{itemize}
\item[1.] $\SFFD(X,Y)=A_X(Y)^\bot;$
\item[2.] $\SFFD \ \text{is a bilinear form on}\ D;$
\item[3.] $\SFFD(X,Y)-\SFFD(Y,X)=T(X,Y)^\bot;$
\item[4.] $(A_X(Y)-\nabla_XY)^\top=A_X(Y)-\nabla_XY.$
\end{itemize}
\end{propn}

\begin{proof}
\
\begin{itemize}
\item[1.] $\SFFD(X,Y)=A_X(Y)^\bot.$ This follows from the definition of $\SFFD$, \eqref{SM2} and, because $D$ is integrable,  $[X,Y]^\bot=0$, for $X,Y \in D.$
\item[2.] $\SFFD \ \text{is a bilinear form on}\ D.$ Clearly from the definition $\SFFD$ is function linear in its first argument, and from 1. above it is also function linear in its second argument since $A_X$ is an endomorphism.
\item[3.] $\SFFD(X,Y)-\SFFD(Y,X)=T(X,Y)^\bot.$ This follows from \eqref{eq:SFFD} and \eqref{torsion} and the integrability of $D.$
\item[4.] $(A_X(Y)-\nabla_XY)^\top=A_X(Y)-\nabla_XY.$ This is because $[X,Y]^\top=[X,Y].$
\end{itemize}

\end{proof}
\begin{rmk}
The result $\SFFD(X,Y)-\SFFD(Y,X)=T(X,Y)^\bot$ holds for each choice of $D'$ complementary to $D$. Of course, in the metric case the orthogonal complement is distinguished.
However, we can see the result that $T$ is the skew part of a bilinear form without explicitly choosing $D'$. For each $\theta\in D^\bot$
\begin{align*}
&\theta(\SFFD(X,Y)-\SFFD(Y,X))=\theta(T(X,Y)^\bot)\\
\iff &\theta(A_X(Y)-A_Y(X))=\theta(T(X,Y)),
\end{align*}
remembering that $\theta(A_X(Y))$ is bilinear in $X$ and $Y$.
\end{rmk}

The following corollary gives an unequivocally geometric representation of the classical second fundamental form.

\begin{cor} Let $S$ be an embedded Riemannian submanifold of a Riemannian manifold $M$ with shape map associated with the Levi-Civita connection on $M$ defined in \ref{shape map}.
If $X,Y$ are vector fields tangent to $S$ then the map
$$(X,Y) \mapsto (A_X(Y))^\bot$$
is function-bilinear and symmetric and equal to the second fundamental form on $S$.

\end{cor}

\section{Codimension $\bold{1}$ generalisations}

Codimension one submanifolds (hypersurfaces) have a special place in Riemannian geometry with the (symmetric) second fundamental form inducing the self-adjoint shape operator whose real eigenvalues are the principal curvatures. We will demonstrate a generalisation to codimension one distributions in which the torsion is the obstruction to the construction of such curvatures.

Again we follow Lee~\cite{Lee97}. Let $S$ be an $n-1$ dimensional submanifold of $(M,g)$ with unit normal $N$ (determined up to a sign). The {\em scalar second fundamental form} $h$ on $S$ is defined by
\begin{equation*}
h(X,Y):=g(\text{II}(X,Y),N) \iff \text{II}(X,Y)=h(X,Y)N
\end{equation*}
As a result we can define the {\em shape operator} $s$ on $S$, a self-adjoint tangent space endomorphism, by
\begin{equation*}
g(s(X),Y)=h(X,Y).
\end{equation*}
Thinking of the lowering action of $g$ we can write this as
\begin{equation}\label{shape op}
Y\hook g\circ s(X)=h(X,Y)
\end{equation}
The eigenvalues of $s$ are (up to a sign) the $n-1$ {\em principal curvatures}, $\kappa_a$ of $S$. In an orthonormal  basis of eigenvectors of $s$
\begin{equation}\label{PCs}
h(X,Y)=\kappa_1X^1Y^1+\dots+\kappa_{n-1}X^{n-1}Y^{n-1},
\end{equation}
from which many good things follow.

Now suppose that the distribution $D$ of section~\ref{sect3} is of codimension one with Frobenius integrable constraint form $\theta$, that is, $\theta(D)=0$ and $d\theta\wedge\theta=0$. Further suppose that $D'=Sp\{N\}$ with $\theta(N)=1$.\newline
Now we define the scalar second fundamental form, $h_D$, on $D$ by
\begin{equation*}
h_D(X,Y):=\theta(\SFFD(X,Y)) \iff \SFFD(X,Y)=h_D(X,Y)N.
\end{equation*}
(Remember $\SFFD(X,Y):=(\nabla_XY)^\bot\in Sp\{N\}.$)
Because of the first part of proposition~\ref{SFFD} we have
\begin{equation*}
h_D(X,Y):=\theta(A_X(Y)^\bot)=\theta(A_X(Y))=A^*_X\theta(Y).
\end{equation*}
Now for an arbitrary form $\phi$ the pullback $A_X^*\phi$ is not function linear in $X \in D$, however $A_X^*\theta$ does have this property because $II_D$ is bilinear by proposition~\ref{SFFD}. So we define $A^*\theta: X\mapsto A_X^*\theta$ and hence
\begin{equation*}
Y\hook A^*\theta(X)=h_D(X,Y)
\end{equation*}
Comparison with~\eqref{shape op} indicates that $\theta\circ A:=A^*\theta$ plays the role of $g\circ s$ in this generalisation. Part 3 of proposition~\ref{SFFD} shows that it is exactly $T(X,Y)^\bot$, equivalently $\theta(T(X,Y))$, which prevents $h_D$ from being symmetric and hence diagonalisable in the form \eqref{PCs}.

\end{document}